\documentclass{article}
\usepackage{amsmath,amstext, amsthm, amssymb, dsfont,amscd,pslatex,fourier,color}
\usepackage[all,poly]{xy}
\usepackage[numbers]{natbib}

\usepackage[pdfstartview=FitH,
            colorlinks,
           linkcolor=reference,
            citecolor=citation,
            urlcolor=e-mail]{hyperref}
\definecolor{e-mail}{rgb}{0,.40,.80}
\definecolor{reference}{rgb}{.20,.60,.22}
\definecolor{citation}{rgb}{0,.40,.80}

\usepackage[totalheight=7.5in, totalwidth=6.25in,paperwidth=7in,paperheight=9.25in]{geometry}

\RequirePackage{color}

\date{}

\definecolor{todo}{rgb}{1,0,0}
\definecolor{answer}{rgb}{0,0,1}

\title{New effective differential Nullstellensatz\thanks{This work was partially supported by the NSF grants CCF-0952591 and DMS-1413859.}}

\author{Richard Gustavson, Marina Kondratieva,  and Alexey Ovchinnikov \vspace{0.05in}\\ \small
CUNY Graduate Center, Department of Mathematics, 365 Fifth Avenue,
New York, NY 10016, USA\\ \small
\href{mailto:rgustavson@gc.cuny.edu}{rgustavson@gc.cuny.edu}  \\ \small 
Moscow State University, Department of Mechanics and Mathematics, Moscow, Russia\\ \small
\href{mailto:kondratieva@sumail.ru}{kondratieva@sumail.ru}\vspace{0.05in} \\ \small 
CUNY Queens College, Department of Mathematics,
65-30 Kissena Blvd, Queens, NY 11367, USA \\ \small
CUNY Graduate Center, Department of Mathematics, 365 Fifth Avenue,
New York, NY 10016, USA\\ \small
\href{mailto:aovchinnikov@qc.cuny.edu}{aovchinnikov@qc.cuny.edu}
\vspace{-0.25in}}

\newtheorem{theo}{Theorem}[section]
\newtheorem{lemma}[theo]{Lemma}

\newtheorem{cor}[theo]{Corollary}

\theoremstyle{definition}
\newtheorem{ex}[theo]{Example}

\newtheorem{rem}[theo]{Remark}

\numberwithin{theo}{section}
\numberwithin{equation}{section}

\def\N{\mathbb{N}}

\def\<{\langle}
\def\>{\rangle}

\def\ord{\operatorname{ord}}

\DeclareMathOperator{\Quot}{\mathrm Quot}

\newcommand{\Qb}{\mathbb{Q}}

\newcommand{\Le}{\leqslant}
\newcommand{\Ge}{\geqslant}

\newcommand{\K}{K}

\begin{document}

\maketitle

\abstract{We show new upper and lower bounds for the effective differential Nullstellensatz  for differential fields of characteristic zero with several commuting derivations. Seidenberg was the first to address this problem in 1956, without giving a complete solution. The first explicit bounds appeared in 2009 in a paper by Golubitsky, Kondratieva, Szanto, and Ovchinnikov, with the upper bound expressed in terms of the Ackermann function. D'Alfonso, Jeronimo, and Solern\'o, using novel ideas, obtained in 2014 a new bound if restricted to the case of one derivation and constant coefficients.  To obtain the bound in the present paper without this restriction, we extend this approach and use the new methods of Freitag and Le\'on S\'anchez and of Pierce  from 2014, which represent a model-theoretic approach to differential algebraic geometry.
}

\section{Introduction}
It is a fundamental problem to determine whether a system $F=0$, $F = f_1,\ldots,f_r$, of polynomial PDEs with coefficients in a differential field $\K$ is consistent, that is, it has a  solution in a differential field containing $\K$. Differential elimination \cite{Boulier2009,Hubert} is an effective method that can answer this question, and its implementations (including {\sc Maple} packages) can handle examples of moderate size if a sufficiently powerful computer is used. The differential Nullstellensatz states that the above consistency is equivalent to showing that the equation $1=0$ is not a differential-algebraic consequence of the system $F=0$. Algebraically, the latter says that $1$ does not belong to the differential ideal generated by $F$ in the ring of differential polynomials.

The complexity of the effective differential Nullstellensatz is not just a central problem in the algebraic theory of partial differential equations but is also a key to understanding the complexity of differential elimination. It is often the case that this leads to substantial improvements in algorithms.
Let $F = 0$ be a system of  polynomial PDEs in $n$ differential indeterminates  (dependent variables) and $m$ derivation operators $\partial_1,\ldots,\partial_m$ (that is, with $m$ independent variables), of total order $h$ and degree $d$. For every  non-negative integer $b$,  let $F^{(b)}=0$ be the set of differential equations obtained from the system $F=0$ by differentiating each equation in it $b$ times with respect to any combination of $\partial_1,\ldots,\partial_m$. An upper bound for the effective differential Nullstellensatz is a numerical function $b(m, n, h, d)$ such that, for all such $F$, the system $F = 0$ is inconsistent if and only if the system of polynomial equations in $F^{(b(m,n,h,d))}$ is inconsistent.
By the usual Hilbert's Nullstellensatz, the latter is equivalent to $$1\in \left(F^{(b(m,n,h,d))}\right).$$ For example, in the system of polynomial PDEs 
\begin{equation}\label{eq:deq}\left\{\begin{array}{ll}
    u_x+v_y=0\\
    u_y-v_x=0 & \hspace{1cm} \\
    (u_{xx}+u_{yy})^2+(v_{xx}+v_{yy})^2=1
  \end{array}\right.
  \end{equation}
  $\partial_1 = \partial/\partial_x$, $\partial_2 = \partial/\partial_y$, and so $m=2$, the differential indeterminates are $u$ and $v$, and so $n=2$, the maximal total order of derivatives is $h=2$, and the maximal total degree is $d=2$.  The corresponding system of polynomial equations is
\begin{equation*}\left\{\begin{array}{ll}
    z_1+z_2=0\\
    z_3-z_4=0  \\
    (z_5+z_6)^2+(z_7+z_8)^2=1
  \end{array}\right.
  \end{equation*}
  which is consistent (e.g., take $z_1=\ldots=z_7=0$ and $z_8=1$).
On the other hand, system~\eqref{eq:deq} is inconsistent. Indeed, applying $\partial_1$ and $\partial_2$ to the first and second equations in~\eqref{eq:deq}, consider the extended system
\begin{equation}\label{eq:deq2}\left\{\begin{array}{ll}
    u_x+v_y=0\\
    u_y-v_x=0\\
        u_{xx}+v_{xy}=0\\
    u_{yy}-v_{xy}=0\\
    u_{xy}+v_{yy}=0\\
    u_{xy}-v_{xx}=0\\
    (u_{xx}+u_{yy})^2+(v_{xx}+v_{yy})^2=1
  \end{array}\right.
\end{equation} It now remains to substitute the sum of the third and fourth equations and the difference of the fifth and sixth equations into the last equation to obtain $0=1$. The equivalent polynomial system is
$$\left\{\begin{array}{ll}
    z_1+z_2=0\\
    z_3-z_4=0\\
        z_5+z_9=0\\
    z_6-z_9=0\\
    z_{10}+z_8=0\\
    z_{10}-v_7=0\\
    (z_5+z_6)^2+(z_7+z_8)^2=1
  \end{array}\right.$$
  which is inconsistent by the above reasoning. In this particular example, it is enough to differentiate the first two equations of~\eqref{eq:deq} only once to discover that the corresponding polynomial system is inconsistent.
  
  Our main result, Theorem~\ref{thm:main}, provides a uniform upper bound on the number of differentiations needed for all systems of polynomial PDEs with the number of derivations, indeterminates, total order, and total degree bounded by $m$, $n$, $h$, and $d$, respectively. This bound substantially outperforms the previously known general upper bound~\cite{DiffNull}.
Our result reduces the problem to the polynomial effective Nullstellensatz, which has been very well studied, with many sharp results  available (see, for example, \cite{Brownawell,Dube2,Jelonek,Kollar,Teresa} and the references given there). On the other hand, note that our problem is substantially more difficult that this problem, because the polynomial effective Nullstellensatz corresponds (see Theorem~\ref{theo:polytodiff}) to the effective differential Nullstellensatz restricted to systems of linear ($d=1$) PDEs in one indeterminate ($n=1$) with constant coefficients, and we do not make these restrictions.

The effective differential Nullstellensatz was first addressed in~\cite{Seidenberg}, without providing a complete solution. In the ordinary case ($m=1$), the first bound, which was triple-exponential in $n$ and polynomial in $d$ appeared in~\cite{Grigoriev}. The first general formula for the upper bound and first series of examples for the lower bound in the case of $m$ derivations appeared in~\cite{DiffNull}. That formula is expressed in terms of the Ackermann function and is primitive recursive but not elementary recursive in $n,h,d$ for each fixed $m$ and is not primitive recursive in $m$. In the case of constant coefficients and $m=1$, an important breakthrough was made in~\cite{DJS}, where a double-exponential bound in $n$ was given. 

In the present paper, we go much beyond the final result of~\cite{DJS} and use the new methods discovered by logicians for fields with several commuting derivations \cite{FS,Pierce} to obtain a new upper bound for the most general case: the coefficients do not have to be constant and we allow any number $m$. For any $m$, our bound is polynomial in $d$. For  $m=1,2$, a more concrete analysis of the bound is given in Section~\ref{sec:concrete}, which shows that our bound is elementary recursive in these cases. In particular, for $m=1$, it is double-exponential in $n$ and $h$ and is polynomial in $d$, as in~\cite{DJS}, but does not require constant coefficients. For $m=2$, it is iterated exponential in $h$ of length $n+2$. Our Examples~\ref{ex:42} and~\ref{ex:46} show lower bounds that are polynomial in $h$ and $d$ and exponential in $mn$.

The paper is organized as follows. We begin in Section~\ref{sec:basicdefs} with introducing the concepts and notation that we further use in the paper. Section~\ref{sec:mainsec} contains the main result of the paper, Theorem~\ref{thm:main}, as well as a discussion of the bound for small numbers of derivations in Section~\ref{sec:concrete}. The lower bound is given in Section~\ref{sec:lowerbound}.
\section{Basic definitions}\label{sec:basicdefs}
A detailed introduction to the subject can be found in~\cite{BuiumDADG,Kap,Kol,MarkerMTDF}. We will introduce only what is used in the paper.
A {\it differential ring} $(\K,\Delta)$ is a commutative ring $\K$ with a finite set $\Delta = \{\partial_1,\ldots,\partial_m\}$ of pairwise commuting derivations on $\K$.  We let $$\Theta = \left\{\partial_1^{i_1}\cdot\ldots\cdot\partial_m^{i_m}\:\big|\: i_j\Ge 0,\ 1\Le j\Le m\right\}.$$ For $\theta= \partial_1^{i_1}\cdot\ldots\cdot\partial_m^{i_m}$, we let $$\ord\theta = i_1+\ldots+i_m.$$ Let also
$$
R=\K\{y_i\:|\: 1\Le i\Le n\} := \K[\theta y_i\:|\: \theta\in\Theta, 1\Le i \Le n]\ \ \text{and}\ \  R_h = \K\big[\theta y_i\:\big|\: 1\Le i\Le n,\ \ord\theta \Le h\big],\quad h\Ge 0.
$$
The ring $R$ defined above is called the {\it ring of differential polynomials} in differential indeterminates $y_1,\ldots,y_n$ and with coefficients in $\K$. The ring $R$ is naturally a differential ring. We will use, what we will call, an {\it orderly ranking} $>$ on $\Theta$. This is a total order on $\Theta$ such that, for all $\theta_1$, $\theta_2 \in \Theta$, if  $\ord\theta_1 > \ord\theta_2$, then $\theta_1 > \theta_2$. An example of such a ranking is given by ordering the $n$-tuples of exponents in $\Theta$ degree-lexicographically.

For a subset $F$ of a ring $R$, $(F)$ denotes the ideal generated by $F$ and $\sqrt{(F)}$ denotes the radical ideal generated by $F$. For a subset $F$ of a differential ring $R$, $[F]$ denotes the differential ideal generated by $F$ in $R$ and $\{F\}$ denotes the radical differential ideal generated by $F$ in $R$. Note that, if $\Qb \subset R$, then $\{F\}=\sqrt{[F]}$.

A field $L$ is called  {\it differentially closed} if, for every $F \subset L\{y_1,\ldots,y_n\}$, the existence of a differential field $M\supset L$ and $(a_1,\ldots,a_n) \in M^n$ such that, for all $f \in F$, $f(a_1,\ldots,a_n)=0$ implies the existence of $(b_1,\ldots,b_n) \in L^n$ with, for all $f \in F$, $f(b_1,\ldots,b_n)=0$. In other words, $L$ is differentially closed if and only if the inconsistency of a system of polynomial differential equations with coefficients in $L$ is preserved under differential field extensions of $L$.

Let $\K$ be a differential field of characteristic zero.
The {\it weak form of the differential Nullstellensatz} states that, for all $F \subset \K\{y_1,\ldots,y_n\}$, $1\notin[F]$ if and only if, for all differentially closed fields $L\supset K$, there exists $(a_1,\ldots,a_n) \in L^n$ such that, for all $f \in F$, $f(a_1,\ldots,a_n)=0$. The {\it strong form of the differential Nullstellensatz} states that for all $F \subset \K\{y_1,\ldots,y_n\}$ and $g \in \K\{y_1,\ldots,y_n\}$, $g\in \sqrt{[F]}$ if and only if, for all differentially closed fields $L\supset K$ and all $(a_1,\ldots,a_n) \in L^n$ such that, for all $f \in F$, $f(a_1,\ldots,a_n)=0$, we have $g(a_1,\ldots,a_n)=0$. 
\section{Main result}\label{sec:mainsec}
We will start by showing several auxiliary results in Section~\ref{sec:prep}. The main result, Theorem~\ref{thm:main}, is contained in Section~\ref{subsec:main}. This is continued with an analysis of our estimate for particular numbers of derivations in Section~\ref{sec:concrete}.
\subsection{Preparation}\label{sec:prep}
Let $l \Ge 0$ and $J\subset R_l$ be an ideal. For each $k \in \N$, let
$J^{(k)}$ be the ideal of the ring $R_{l+k}$ generated by the derivatives of the elements of $J$ up to order $k$ (cf. \cite{MoosaScanlon}), that is,
$$J^{(k)}=\big(\theta g\:|\: g\in I,\:\ord\theta\Le k\big).$$  For $D\in\Theta$,  let
$J^{(D)}$ be the ideal of $R_{l+\ord D}$ generated by the derivatives of the elements of $J$ not exceeding $D$ in an orderly ranking, that is,
$$J^{(D)}=\big(\theta g\:|\: g\in I,\: \theta\Le D\big).$$
For every ideal $J$ of the ring $R_l$,  we let $$J'=\sqrt{(\theta J\:|\:\ord \theta\Le 1)}\cap R_l.$$ 
We also let  
$$\alpha_l=\binom{l+m}{m}.$$ 
Note that $$\dim_{\K}R_l = n\alpha_l.$$
\begin{lemma}\label{lem:1}
Let    $J\subset R_l$ be an ideal, $p >0$, and $\big(J'\big)^ p\subseteq J^{(1)}$. Then, for all $k\Ge 0$,
$$\sqrt{J'^{(k)}}
 \subseteq \sqrt{J^{(kp+1)}}.$$ 
 \end{lemma}
\begin{proof}
Fix an orderly ranking on the ring of $\Delta$-polynomials $\K\{y\}$. Let  $D\in \Theta$, $p\in N$. Then there exist  $\beta_l\in \K\{y\}$ such that
\begin{equation}\label{eq:1}
D^p \big(y^p\big)=(D(y))^p+\sum_{\theta(l)y<Dy}\beta_l\theta(l)y.
\end{equation} 
Indeed, let $D=\partial_1^{i_1}\ldots\partial_m^{i_m}\in \Theta(r)$.
By the Leibniz rule, for every weight- and degree-homogeneous differential polynomial $z$, the differential polynomial $\partial z$ is homogeneous of degree equal to $\deg z$ and of weight with respect to $\partial$ equal to that of $z$ plus one.
Hence,
\begin{equation}\label{eq:2}
D^p\big(y^p\big)=\sum_{\sum_{k=1}^p l_1^k= pi_1,\ldots,\sum_{k=1}^p l_m^k= pi_m}
c_{?}\partial_1^{l_1^1}\ldots\partial_m^{l_m^1}y\cdot \ldots
\cdot\partial_1^{l_1^p}\ldots\partial_m^{l_m^p}y,
\end{equation}
where $c_?$ are some elements of $\K$. 
Consider a monomial in the right-hand side of~\eqref{eq:2}. Suppose that it is of order greater than $r$ in every differential indeterminate that appears in it.
Then, for each $k$, $1\Le k\Le p$, we have 
 $$l_1^k+\ldots+ l_m^k>r.$$ 
Adding $p$ inequalities, we obtain 
$$pr=p(i_1+\ldots + i_m)=\sum_{k=1}^p\sum_{t=1}^ml_t^k> pr,$$ which is a contradiction. 
Therefore, for each monomial in the right-hand side of~\eqref{eq:2}, one of the factors has order $\Le r$, and we have:
\begin{equation}\label{eq:3}
D^p \big(y^p\big)=
\sum_{{\sum_{k=1}^p l_1^k= pi_1,\ldots,\sum_{k=1}^p l_m^k= pi_m}\atop
 {\sum _{k=1}^m l_k^1\Ge r,\ldots,   \sum _{k=1}^m l_k^p\Ge r}}
c_{?}\partial_1^{l_1^1}\ldots\partial_m^{l_m^1}y\cdot \ldots
\cdot\partial_1^{l_1^p}\ldots\partial_m^{l_m^p}y
+\sum_{l<r}\beta_l\theta(l)y,
\end{equation} 
If, in a monomial from the first sum in~\eqref{eq:3}, at least one of the factors had order greater than $r$, then, as in the above, by adding $p$ inequalities, we would arrive at a contradiction. Thus, we obtain:
\begin{equation}\label{eq:4}
D^p \big(y^p\big)=
\sum_{{\sum_{k=1}^p l_1^k= pi_1,\ldots,\sum_{k=1}^p l_m^k= pi_m}\atop
{ \sum _{k=1}^m l_k^1= r,\ldots,   \sum _{k=1}^m l_k^p= r}}
c_{?}\partial_1^{l_1^1}\ldots\partial_m^{l_m^1}y\cdot \ldots
\cdot\partial_1^{l_1^p}\ldots\partial_m^{l_m^p}y
+\sum_{l<r}\beta_l\theta(l)y,
\end{equation} 
Let the ranking be such that $\partial_1>\ldots>\partial_m$ and, in the first sum in~\eqref{eq:4}, for one of the factors, we have  $l_1^k> i_1$ for all $k$,  $1\Le k\Le p$.
Adding these $p$ inequalities, we obtain $$pi_1=\sum_{k=1}^pl^k_1>pi_1,$$ which gives a contradiction. Thus,
$$
D^p \big(y^p\big)=
\sum_{{{\sum_{k=1}^p l_1^k= pi_1,\ldots,\sum_{k=1}^p l_m^k= pi_m}
\atop
 {\sum _{k=1}^m l_k^1= r,\ldots,   \sum _{k=1}^m l_k^p= r}}
 \atop
{l_1^1=i_1,\ldots,l_1^p=i_1}
}
c_{?}\partial_1^{l_1^1}\ldots\partial_m^{l_m^1}y\cdot \ldots
\cdot\partial_1^{l_1^p}\ldots\partial_m^{l_m^p}y
+\sum_{\theta(l)y<Dy}\beta_l\theta(l)y,
$$ 
As before, note that, for each monomial from the first sum, one cannot have   $l_2^k> i_2$ for all $k$, $1\Le k\Le p$.
Therefore, in this sum, we are left with just  the monomials of order $\Le r$ of the form 
$$\partial_1^{i_1}\partial_2^{i_2}\partial_3^{l_3^1}\ldots \partial_m^{l_m^1}y\cdot
\partial_1^{i_1}\partial_2^{i_2}\partial_3^{l_3^p}\ldots\partial_m^{l_m^p}y ,$$
 moving the rest of the monomials to the other sum.
We now see that, distributing all monomials between these two sums accordingly, we obtain that the first sum contains only one summand and, therefore, obtain~\eqref{eq:1}.

We will prove the statement of the lemma now.  By induction on $k$, we will show that, for all $D\in \Theta$ of order $k$,
$$\sqrt{J'^{(D)}}
 \subseteq \sqrt{J^{(kp+1)}}.$$ 
For every $j'\in J'$, by the definition of $J'$, there exists $p \Ge 1$ such that 
\begin{equation}\label{eq:6}
j'^p=\theta j_0,\quad j_0 \in J,\quad\ord\theta \Le 1.
\end{equation}
Let $k=0$, $D\in \Theta(0)$. By the definition of 
$J'$, $$\sqrt{J'}=J'\subseteq \sqrt{J^{(1)}},$$
and the statement holds.
Consider an orderly ranking $<$  on the set of differential polynomials in one differential indeterminate and 
let, for all $D'<D$, $\ord D=k$, 
$$\sqrt{J'^{(D')}}
 \subseteq \sqrt{J^{(kp+1)}}.$$ 
By~\eqref{eq:1} for $y=j'$,~\eqref{eq:6} implies 
\begin{equation}\label{eq:8}
\big(D\big(j'\big)\big)^p+\sum_{D'< D}\beta_{D'} D'j'=D^p
\big(\theta(1)j_0\big).
\end{equation} 
Since $$D^p\big(\theta(1)j_0 \big)\in  J^{(kp+1)} $$ and, by the inductive hypothesis, a zero of $J^{(kp+1)}$ vanishes 
$D' j'$ in~\eqref{eq:8}, we have
 $$\sqrt{J'^{(D)}}\subseteq \sqrt{J^{(kp+1)}}.\qedhere$$ 
\end{proof}

\begin{lemma}\label{lem:order}
Let $s \Ge 0$,$$ I_0={(F_0)}\subseteq I_1=(F_1)\subseteq ...\subseteq I_s={(F_s)}\subseteq R_l$$
be ideals of $R_l$, $p_j \Ge 0$, $0\Le j\Le s$, and, for all $i$, $1\Le i\Le s$,
$$I_i=\big(I_{i-1}\big)'\quad\text{and}\quad  I_i^{p_i}\subset I_{i-1}^{(1)}.$$ 
Then, for all $q \in \N$, there exists $k$ such that  
$$ I_s^{(q)}\subseteq\sqrt{(F_0)^{(k)}}\quad \text{and}\quad k\Le 1+p_1+p_1\cdot p_2+\ldots+q\cdot p_1\cdot\ldots\cdot p_s.
$$
\end{lemma}
 \begin{proof}
Let  $g\in I_s^{(q)}$. Then  $g\in I_{s-1}'^{(q)}$.  Set $J=I_{s-1}$. Then
$$J^{(1)}=\big(F_{s-1},\partial F_{s-1}\:\big|\:\partial\in\Delta\big),\quad J'=\sqrt{J^{(1)}}\cap R_l.$$ 
 Applying Lemma~\ref{lem:1} with $k=q$,
we obtain
$$g\in \sqrt{I_{s-1}^{(qp_s+1)}}.$$  
Again, by Lemma~\ref{lem:1} with $k=qp_s+1$ and $J=I_{s-1}$, we have  
$$g\in \sqrt{I_{s-2}^{(p_{s-1}(qp_s+1)+1)}}.$$
 Arguing similarly, we obtain
$$g\in \sqrt{I_{0}^{(1+\ldots(1+(1+p_sq)p_{s-1})\ldots p_1)}}\subseteq 
\sqrt{(F_0)^{(1+\ldots(1+(1+qp_s)p_{s-1})\ldots p_1)}}.\qedhere$$
\end{proof}

For all $F\subset R_h$, we let (see \cite[pages 15-16]{FS} for the recursive definition of $T_h^{m,n}$, which we do not give here, because it is lengthy) \begin{equation}\label{eq:IT}\quad I=\sqrt{(F)},\quad T=T_h^{m,n},\quad 
I_0=\sqrt{ I^{(T)}}\cap R_{T-1}
\end{equation}
and 
\begin{equation}\label{eq:Ik}
I_k=\sqrt{\big( g,\partial g\:|\:g\in 
I_{k-1},\partial\in\Delta\big)}\cap R_{T-1} = \sqrt{I_{k-1}^{(1)}}\cap R_{T-1}.
\end{equation}

\begin{lemma}[{cf. \cite[Proposition~4.1]{FS}}]\label{lem:dimdrop}
If $1\in[F]$, then, for all $k \Ge 1$ such that $I_k\ne R_{T-1}$,  $$\dim I_{k-1}>\dim I_k.$$ 
\end{lemma}
\begin{proof}
Suppose that
\begin{equation}\label{eq:eqdim}
\dim I_k=\dim I_{k-1}
\end{equation} for some $k\Ge 1$. Fix such $k$. Since $I_{k-1}\subset I_k$, by~\eqref{eq:eqdim}, there exists a minimal prime component of $I_k$ that is a minimal prime component of $I_{k-1}$. Pick such a component and denote it by $Q$.
Let $P$ be a prime component of
$\sqrt{I_{k-1}^{(1)}}\subset R_T$ such that 
\begin{equation}\label{eq:PQ}
Q = P\cap R_{T-1},
\end{equation}
which exists by \cite[Proposition 16, Section 2, Chapter II]{Bourbaki:commutativealgebra}.
Let
$$R_T\big/P = \K\left[a_i^\theta \:\Big|\:1\Le i\Le n,\ \ord\theta\Le T\right].$$ Then, by~\eqref{eq:PQ},
$$R_{T-1}\big/Q = \K[b],\quad b:=\left(a_i^\theta\: \Big|\: 
1\Le i\Le n,\ \ord\theta< T\right).$$
We will show that the field $$L=K\left(a_i^\theta\:\Big|\:1\Le i\Le n,\ \ord\theta\Le T\right)$$ 
satisfies the differential condition \cite[page~10]{Pierce}. Indeed, 
it is sufficient to show that, for all $k$, $1\Le k\Le m$, there exists a derivation  $D_k:L'\to L$,
 where $L'=K(b)$, extending $\partial_k$ so that $$D_ka_i^\theta=a_i^{\partial_k\cdot\theta}$$ 
for all $a_i^\theta\in L'$.
For this, it is sufficient to show that, if  $$f\in K\left[x_i^\theta\:\Big|\:
1\Le i\Le n,\ \ord\theta\Le T-1\right]\quad\text{and}\quad f(b)=0,$$ then
\begin{equation}\label{eq:diffvanish}
\sum_{1\Le i\Le n\atop\ord\theta\Le T-1}\frac{\partial f}
{\partial x_i^\theta}(b)a_i^{\partial_k\cdot\theta}+f^{\partial_k}(b)=0
\end{equation}
(here $f^{\partial_k}$ is the polynomial obtained from $f$ by applying $\partial_k$ to its coefficients.)
Note that $Q^{(1)}\subset P$. Indeed, let $J$ be the intersection of all minimal prime components of $I_{k-1}$ not equal to $Q$, $h \in Q$, $\partial\in\Delta$, and $g \in J\setminus Q$ be such that $hg \in I_{k-1}$. By \cite[Lemma~1.3, Chapter~I]{Kap}, $$g\cdot\partial h \in \sqrt{I_{k-1}^{(1)}}\subset P.$$ Since $g \notin Q$ and $g \in R_{T-1}$, $g \notin P$. Hence, $\partial h \in P$.

Finally, since $f(b)=0$, $f\in Q$. Hence, $\partial f\in Q^{(1)}\subset P$,
which implies~\eqref{eq:diffvanish}.
By the choice of $T$ and \cite[Theorem~4.10]{Pierce}, if  $L$ satisfies the differential condition, 
 $\Quot(R/\{Q\})$ is a non-trivial extension of the differential field
 $(\K,\Delta)$, which contradicts $1\in[F]\subset\big[I_{k}\big]$.   
\end{proof}
\subsection{Main result}\label{subsec:main}

\begin{theo}\label{thm:main}
Let  $h,D\Ge 0$, $F\subset R_h$,  $\deg F \Le D$. Then $1\in [F]$ if and only if
there exists $k\Ge 0$ such that $$k \Le \left(n\alpha_{T-1}D\right)^{2^{O\left(n^3\alpha_T^3\right)}}\quad \text{and}\quad
1\in (F)^{(k)},$$
where $\alpha_T=\binom{T+m}{m}$ and $T=T^{m,n}_h$ is recursively defined in~\cite[pages~15-16]{FS}. 
\end{theo}
\begin{proof}
If $1 \in (F)^{(k)}$, then $1 \in [F]$ by definition. We will now show the reverse implication.
Let $s = \dim Z(F)$ and also
\begin{equation}\label{eq:abc}
a := n\alpha_{T-1},\quad b := n\alpha_T,\quad c := O\left(n^2\alpha_{T-1}^2\right).
\end{equation}
 Then $$\deg Z(F)=\deg Z(I) \Le D^s\Le D^{n\alpha_h}.$$ Hence, by \cite[Proposition~4]{DJS}, the ideal $I$ (as well as the ideal $I^{(T)}$) can be generated by polynomials of degree at most 
$$
\left(n\alpha_hD^s\right)^{2^{O\left(sn\alpha_h\right)}}\Le \left(n\alpha_h\right)^{2^{O\left(n^2\alpha_h^2\right)}}D^{n\alpha_h\cdot 2^{O\left(n^2\alpha_h^2\right)}}=\left(n\alpha_hD\right)^{2^{O\left(n^2\alpha_h^2\right)}}=:d_F.
$$
Then $$
\deg Z(I_0) \Le   d_F^{n\alpha_T}= \left(n\alpha_hD\right)^{b 2^{O\left(n^2\alpha_h^2\right)}}=: D_0
$$
and the ideal $I_0$ can be generated by polynomials of degrees at most
$$
\left(aD_0\right)^{2^c}=
\left(a\right)^{2^c}\left(n\alpha_hD\right)^{b2^{O\left(n^2\alpha_h^2\right)+c}} =:d_0.
$$

Moreover, by~\cite[Theorem~1.3]{Jelonek},
$$
\sqrt{I^{(T)}}^{p_0}\subset I^{(T)},\quad p_0 := d_F^b.
$$
Hence, $$I_0^{p_0}\subset I^{(T)}\cap R_{T-1}.$$
Continuing this way, we obtain that
$$
\deg Z(I_{i+1}) \Le d_i^{n\alpha_T} =: D_{i+1}
$$
and the ideal $I_{i+1}$ can be generated by polynomials of degrees at most
$$
d_{i+1} = a^{2^c}(d_i)^{b2^c}=a^{2^c}\left(a^{2^c}d_{i-1}^{b2^c}\right)^{b2^c}=a^{2^c+b2^{2c}}d_{i-1}^{b^22^{2c}}=a^{2^c+b2^{2c}}\left(a^{2^c}d_{i-2}^{b2^c}\right)^{b^22^{2c}}=a^{2^c+b2^{2c}+b^22^{3c}}d_{i-2}^{b^32^{3c}}.
$$
Therefore,
$$
d_{i+1} = a^{2^c\sum\limits_{j=0}^q \left(b2^c\right)^j}d_{i-q}^{\left(b2^c\right)^{q+1}} =a^{2^c\sum\limits_{j=0}^i \left(b2^c\right)^j}d_0^{\left(b2^c\right)^{i+1}}=a^{2^c\frac{\left(b2^c\right)^{i+1}-1}{b2^c-1}}d_0^{\left(b2^c\right)^{i+1}}\Le \left(a^{2^c}d_0\right)^{\left(b2^c\right)^{i+1}}
$$
and
$$
D_{i+1} \Le \left(a^{2^c}d_0\right)^{b\left(b2^c\right)^i}.
$$
Again, by~\cite[Theorem~1.3]{Jelonek}, 
$$
\sqrt{I_i^{(1)}}^{p_{i+1}}\subset I_i^{(1)},\quad p_{i+1} := \left(a^{2^c}d_0\right)^{b\left(b2^c\right)^i} \Ge d_i^b,\quad i\Ge 0.
$$
Hence,
$$
I_{i+1}^{p_{i+1}} \subset I_i^{(1)}\cap R_{T-1}.
$$
By Lemma~\ref{lem:dimdrop}, $1 \in I_a$. By Lemma~\ref{lem:order} applied to~\eqref{eq:Ik}, 
$$
1 \in I_0^{\left(1+p_1+p_1\cdot p_2+\ldots+p_1\cdot\ldots\cdot p_a\right)}.
$$
Again by Lemma~\ref{lem:order}, for all $q \Ge 0$,
$$
I_0^{(q)} \subset \sqrt{I^{\left(T+1+qp_0\right)}}.
$$
Hence,
$$
1 \in I^{\left(T+1+p_0\left(1+p_1+p_1\cdot p_2+\ldots+p_1\cdot\ldots\cdot p_a\right)\right)} .
$$ 
By Lemma~\ref{lem:1} again, we obtain
\begin{equation}\label{eq:p0Tpa}
1 \in (F)^{\left(p_0\left(T+1+p_0\left(1+p_1+p_1\cdot p_2+\ldots+p_1\cdot\ldots\cdot p_a\right)\right)\right)}=(F)^{\left(p_0\left(T+1+\ldots(1+(1+p_a)p_{a-1})\ldots p_0\right)\right)}=(F)^{(p_0(T+p_0\cdot\ldots\cdot p_a))},
\end{equation}
with the latter equality following from the definition of $c$ via the $O$-symbol.
Note that
\begin{align*}
p_0^2p_1\cdot\ldots\cdot p_a &= d_F^{2b}\left(a^{2^c}d_0\right)^{b \sum\limits_{j=0}^{a-1}\left(b2^c\right)^j}=d_F^{2b}\left(a^{2^c}d_0\right)^{b \left(\frac{\left(b2^c\right)^a-1}{b2^c-1}-1\right)}\\&\Le d_F^{2b}\left(a^{2^c}d_0\right)^{b\left(b2^c\right)^a}= d_F^{2b}\left(a^2d_F^b\right)^{\left(b2^c\right)^{a+1}}\\
&\Le\left(ad_F\right)^{b\left(\left(b2^c\right)^{a+1}+2\right)} =\left(ad_F\right)^{2^{cb}}=a^{2^{cb}}\left(n\alpha_hD\right)^{2^{O\left(n^2\alpha_h^2\right)+cb}}\\
&=a^{2^{cb}}\left(n\alpha_hD\right)^{2^{cb}}=(aD)^{2^{cb}},
\end{align*}
(again, the equalities hold because of the $O$-definition of $c$ and also because $\alpha_T \Ge 1$ if $h\Ge 1$ and $k=a=0$ if $h=0$)
and the result now follows by substituting~\eqref{eq:abc} in the above and using~\eqref{eq:p0Tpa}.
\end{proof}
\begin{cor}(cf. \cite[Corollary~21]{DJS}) Let  $h,D\Ge 0$, $F\subset R_h$, $f \in R_h$, and  $\max\{\deg f, \deg F\} \Le D$.
Then $f \in \sqrt{[F]}$ if and only if there exists $k\Ge 0$ such that $$k \Le \left(n\alpha_{T'-1}D\right)^{2^{O\left(n^3\alpha_{T'}^3\right)}}\quad \text{and}\quad
f\in \sqrt{(F)^{(k)}},$$
where $T' := T_h^{m,n+1}$.
\end{cor}
\begin{proof} If $f\in \sqrt{(F)^{(k)}}$, then $f \in \sqrt{[F]}$ by definition. Let $f \in \sqrt{[F]}$. Then $1 \in[1-tf,F] \subset \K\{y_1\ldots,y_n,t\}$. By Theorem~\ref{thm:main}, $$1 \in \left(\left(1-tf\right)^{(k)},F^{(k)}\right),$$
for which we used the properties of $O$ to go down from $D+1$ (which appears because $\deg tf = \deg f +1$) to $D$ and from $n+1$ to $n$ outside of $T'$. As usual, by substituting $1/f$ into $t$ and clearing out the denominators, we obtain the result.
\end{proof}
\begin{rem}Note that, for $m\Ge 2$, $T \ne O(T')$ (see Section~\ref{sec:concrete}), and so we do not replace $T'$ by $T$ in the corollary. However, for $m=1$, we simply have $T='T=h$.
\end{rem}
\subsection{Concrete values of the number of derivations}\label{sec:concrete}
According to~\cite[page~16]{FS}, if $m=1$, then $T=h$. Then the bound from Theorem~\ref{thm:main} is $$(nhD)^{2^{O\left(n^3(h+1)^3\right)}}$$ and is better than the bound from~\cite[Corollary~19]{DJS}. Indeed,  our result holds for non-constant coefficients and also, if $h=0$, then our bound naturally gives $0$, but the bound from~\cite[Corollary~19]{DJS} gives $(nD)^{2^{O\left(n^3\right)}}$. Of course, one can just prove directly that, if $h=0$, then $k=0$.
If $m=2$, by \cite[Lemma~3.8]{FS}, 
$$
T=T(n,h)=2^{b_{n,h}+1}h,\ \ b_{0,h} = 0,\ b_{i+1,h} = 2^{b_i+1}h+b_i+1.
$$
Therefore, in this case, the bound from Theorem~\ref{thm:main} is polynomial in $D$ (as it is for arbitrary $m$, $n$, and $h$) and is iterated-exponential in $h$, with the length of the tower being equal to $n+2$. 

For comparison, note that the bound from~\cite[Theorem~1]{DiffNull}, $A(m+8,\max(n,h,d))$, has a substantially higher growth rate, as, for example, $A(3,x)$ is exponential in $x$ and $A(4,x)$ is a tower of exponentials of length $x+3$, and the minimal possible value here, $A(9,1)$, is out of reach for any existing computer even to output. 
\section{Lower bound}\label{sec:lowerbound}
The examples in~\cite{DiffNull} show that the lower bound for the effective differential Nullstellensatz is exponential in the number of variables and the number of derivations and polynomial in the degree of the system.  We expand on these results, first by observing how the order of the system affects the lower bound.

\begin{ex}
Consider the system $F = \left\{y_1^d,y_1-y_2^d,\ldots,y_{n-1}-y_n^d,1-y_n^{(h)}\right\} \subset K\{y_1,\ldots,y_n\} =: R$ with one derivation.  A particular and essential case of this, $h=1$, was considered in an unpublished manuscript by York Kitajima, and the argument in the present example is based on Kitajima's argument and extends it, with extra subtleties. Recall that for $s \Ge 2$ and $m,m_1,\ldots, m_s \in \N$ with $m_1 + \ldots + m_s = m$, the multinomial coefficient is
$$
{m \choose m_1,\ldots,m_s} = \frac{m!}{m_1!\cdot \ldots \cdot m_s!}.
$$
For $l \Ge 1$, denote by $M_l$ the multinomial coefficient
$$
M_l = {d^lh \choose d^{l-1}h,\dots,d^{l-1}h},
$$
where this multinomial coefficient contains $d$ terms.  We claim that $(F)^{(j)} \subset I_j$ where
\begin{align*}
I_0 &= \left(y_1,y_2,\ldots,y_{n-1},y_n,1-y_n^{(h)}\right) \\
 I_j &= \left(I_{j-1},y_1^{(j)},y_2^{(j)},\ldots,y_{n-1}^{(j)},y_n^{(j)},y_n^{(h+j)}\right) \hskip3.1cm 1 \Le j \Le h-1 \\
I_j &= \left(I_{j-1},y_1^{(j)},\ldots,y_{n-i}^{(j)} - \prod_{l = 1}^i M_l^{d^{i-l}},y_{n-i+1}^{(j)},\ldots, y_{n-1}^{(j)},y_n^{(h+j)}\right) \hskip.5cm j = d^ih, 1 \Le i \Le n-1 \\
I_j &= \left(I_{j-1},y_1^{(j)},\ldots,y_{n-1}^{(j)},y_n^{(h+j)}\right) \hskip4.35cm \text{otherwise, } j \le d^nh-1.
\end{align*}
Indeed, we can show this by induction on $j$.  The base case $j = 0$ is clear.  Now assume $(F)^{(k)} \subset I_k$ for all $k < j$.  By induction, we only need to show the inclusion of unmixed monomials, i.e. powers of a single derivative of a $y_i$.  The generalized Leibniz rule says that for all $s \Ge 1$, $m\Ge 0$, and $f_1,\ldots,f_s \in R$,
$$
\left(\prod_{r = 1}^s f_r\right)^{(m)} = \sum_{m_1 + \ldots + m_s = m}{m \choose m_1,\ldots,m_s} \prod_{r = 1}^s f_r^{(m_r)}
$$
In our system $F$, we have $s = d$.  Unmixed monomials thus occur when $m_1 = \ldots = m_d = m/d$.  When $j \neq d^ih$, $y_i^{(j/d)} \in I_j$ for all $i$, $1 \Le i \Le n$, so there is nothing to prove.  The case we must consider is when $j = d^ih$, in which case each $m_\alpha$ in the multinomial coefficient is $d^{i-1}h$ and $y_{n-i}^{(d^ih)} \notin I_{d^ih}$.

The only differential polynomial in $(F)^{(d^ih)}$ for which it remains to show that it is in $I_{d^ih}$ is $\left(y_{n-i} - y_{n-i+1}^d\right)^{(d^ih)}$, since $y_{n-i+1}^{(d^{i-1}h)} \notin I_{d^{i-1}h}$ by construction.  Observe that
$$
\left(y_{n-i} - y_{n-i+1}^d\right)^{(d^ih)} = y_{n-i}^{(d^ih)} - M_i\left(y_{n-i+1}^{(d^{i-1}h)}\right)^d + g,
$$
where $g \in K\{y\}$ contains no unmixed monomials, and so is in $I_{d^ih}$.  Thus, it suffices to show that
$$
y_{n-i}^{(d^ih)} - M_i\left(y_{n-i+1}^{(d^{i-1}h)}\right)^d \in I_{d^ih}.
$$
By construction, 
$$
y_{n-i+1}^{(d^{i-1}h)} - \prod_{l = 1}^{i-1}M_l^{d^{i-l-1}} \in I_{d^{i-1}h} \subset I_{d^ih}.
$$
We can thus write
\begin{multline*}
y_{n-i}^{(d^ih)} - M_i\left(y_{n-i+1}^{(d^{i-1}h)}\right)^d = \\
\left(y_{n-i}^{(d^ih)} - \prod_{l = 1}^i M_l^{d^{i-l}}\right) - M_i \sum_{\alpha = 0}^{d-1} \left[\left(\prod_{l = 1}^{i-1} M_l^{d^{i-l-1}}\right)^\alpha\left(y_{n-i+1}^{(d^{i-1}h)}\right)^{d-1-\alpha}\right]\left(y_{n-i+1}^{(d^{i-1}h)} - \prod_{l = 1}^{i-1}M_l^{d^{i-l-1}}\right)
\end{multline*}
in terms of elements of $I_{d^ih}$, completing the induction step and proving that $F^{(j)} \subset I_j$ for all $j$, $1 \Le j \Le d^nh - 1$.

Since $1 \notin I_j$ for all $j$, $0\Le j\Le d^nh-1$, then $1 \notin (F)^{(d^nh-1)}$.  Observe that
\begin{align*}
\left(y_n^{d^n}\right)^{(d^nh)} &= \left(\left(y_n^{d^n}\right)^{(h)}\right)^{((d^n-1)h)} = \left(\left(\sum_{n_{1,1} + \ldots + n_{1,d^n} = h} {h \choose n_{1,1},\ldots,n_{1,d^n}} \prod_{i = 1}^{d^n} y_n^{(n_{1,i})}\right)^{(h)}\right)^{((d^n-2)h)}\\
&= \left(\left(\sum_{n_{1,1} + \ldots + n_{1,d^n} = h}\sum_{n_{2,1} + \ldots + n_{2,d^n} = h} {h \choose n_{1,1},\ldots,n_{1,d^n}}{h \choose n_{2,1},\ldots,n_{2,d^n}} \prod_{i = 1}^{d^n} y_n^{(n_{1,i} + n_{2,i})}\right)^{(h)}\right)^{((d^n-3)h)}\\
&= \ldots = \sum_{n_{1,1} + \ldots + n_{1,d^n} = h} \ldots \sum_{n_{d^n,1} + \ldots + n_{d^n,d^n} = h} \left( \prod_{j = 1}^{d^n} {h \choose n_{j,1},\ldots,n_{j,d^n}}\prod_{i = 1}^{d^n} y_n^{(n_{1,i} + \ldots + n_{d^n,i})}\right).
\end{align*}
Since $y_n^{(h)} \equiv 1$ modulo the system $F$, then $y_n^{(l)} \equiv 0$ for all $l > h$, so the only non-zero terms in this sum will be powers of $y_n^{(h)}$.  We thus have that, modulo $F$, $\left(y_n^{d^n}\right)^{(d^nh)} \equiv 1$, so $1 \in (F)^{(d^nh)}$ and $1 \notin (F)^{(d^nh-1)}$.
\end{ex}

\begin{ex}\label{ex:42}
Consider the following collections of differential polynomials in $K\{y_1,\ldots,y_n\}$ with derivatives $\Delta = \{\partial_1,\ldots,\partial_m\}$, with $d,h \Ge 1$:
\begin{align*}
G_1 &= \left\{(\partial_1y_1)^d,\partial_1y_1 - (\partial_2y_1)^d,\ldots,\partial_{m-1}y_1-(\partial_my_1)^d\right\} \\
G_i &= \left\{\partial_my_{i-1}-(\partial_1y_i)^d,\partial_1y_i - (\partial_2y_i)^d,\ldots,\partial_{m-1}y_i-(\partial_my_i)^d\right\} \hskip.5cm 1 \Le i \Le n-1 \\
G_n &= \left\{\partial_my_{n-1}-(\partial_1y_n)^d,\partial_1y_n-(\partial_2y_n)^d,\ldots,\partial_{m-1}y_n-(\partial_my_n)^d,1-\partial_m^{h+1}y_n\right\}.
\end{align*}
Similar to what is done in~\cite{DiffNull}, if we replace $F$ in the previous example by $G = \bigcup_{i = 1}^n G_i$, then the elements of $G$ will need to be differentiated a minimum of $d^{mn}h$ times in order to reduce the system to 1, so $1 \in (G)^{(d^{mn}h)}$ and $1 \notin (G)^{(d^{mn}h-1)}$.
\end{ex}

In these examples, we see that that the lower bound for having $f \in (G)^{(k)}$ is exponential in the number of derivations and number of variables and linear in the order of the system.  However, these known examples are non-linear.  We will use the lower bound on the effective polynomial Nullstellensatz to construct an example of a linear system $G \subset K\{y_1,\ldots,y_n\}$ with $f \in (G)^{(k)}$ but $f \notin (G)^{(k-1)}$, where $k$ is exponential in the number of derivations and the number of variables and polynomial in the order of the system.

We use a system of polynomials to construct a system of differential polynomials.  We begin with polynomials in $K[X_1,\ldots,X_m]$ and construct differential polynomials in $K\{y\}$ with derivations $\Delta = \{\partial_1,\ldots,\partial_m\}$, where $K$ is constant with respect to each $\partial_i$.  Given $\alpha = (\alpha_1,\ldots,\alpha_m) \in \N^m$, denote $X^\alpha = X_1^{\alpha_1}\cdot\ldots\cdot X_m^{\alpha_m}$ and $\partial^\alpha = \partial_1^{\alpha_1}\cdot\ldots\cdot \partial_m^{\alpha_m}$.

Suppose we have $f_1,\ldots,f_r \in k[X_1,\ldots,X_m]$.  For each $i$, $1 \Le i \Le r$, there exist $\alpha_{i,1},\ldots,\alpha_{i,N_i} \in \N^m$ and $c_{i,1},\ldots,c_{i,N_i} \in K$ such that
$$
f_i = \sum_{j = 1}^{N_i} c_{i,j} X^{\alpha_{i,j}}.
$$
We then define $\tilde{f}_i \in K\{y\}$ to be
$$
\tilde{f}_i = \sum_{j = 1}^{N_i} c_{i,j} \partial^{\alpha_{i,j}}y.
$$
Similarly, given $\displaystyle f = \sum_{j = 1}^N s_jX^{\gamma_j} \in K[X_1,\ldots,X_m]$, we can define $\displaystyle \tilde{f} = \sum_{j = 1}^N s_j \partial^{\gamma_j}y \in K\{y\}$.  Consider the system $G = \{\tilde{f}_1,\ldots,\tilde{f}_r\}$.

\begin{theo}\label{theo:polytodiff}
Let $f,f_1,\ldots,f_r \in K[X_1,\ldots,X_m]$, $\tilde{f},\tilde{f}_1,\ldots,\tilde{f}_r \in K\{y\}$, and $G \subset K\{y\}$ be defined as above.  Suppose $f \in (f_1,\ldots,f_r)$ and let $k$ be the lower bound for the degree of the coefficients of the $f_i$ in any possible representation of $f$.  Then $\tilde{f} \in (G)^{(k)}$ but $\tilde{f} \notin (G)^{(k-1)}$.
\end{theo}

\begin{proof}
Suppose $f \in (f_1,\ldots,f_r)$, so there exist $g_1,\ldots,g_r \in K[X_1,\ldots,X_m]$ such that $f = g_1f_1 + \ldots + g_rf_r$.  As with the $f_i$s, there exist $\beta_{i,j} \in \N^m$ and $d_{i,j} \in K$, $j = 1,\ldots,M_i$, such that we can write each $g_i$ as
$$
g_i = \sum_{j = 1}^{M_i}d_{i,j}X^{\beta_{i,j}}.
$$
It is then easy to see that
\begin{equation}\label{eq:polytodiff}
\sum_{j = 1}^{M_1}d_{1,j}\partial^{\beta_{1,j}}\left(\tilde{f}_1\right) + \ldots + \sum_{j = 1}^{M_r}d_{r,j}\partial^{\beta_{r,j}}\left(\tilde{f}_r\right) = \sum_{j = 1}^N s_j \partial^{\gamma_j}y = \tilde{f}.
\end{equation}
Since $G = \{\tilde{f}_1,\ldots,\tilde{f}_r\}$, we thus have that $\tilde{f} \in [G]$, and since the maximum degree of the $g_i$s is $k$, the maximum order of the $\partial^{\beta_{i,j}}$s is also $k$, so $\tilde{f} \in (G)^{(k)}$.

It remains to show that $\tilde{f} \notin (G)^{(k-1)}$.  Suppose for a contradiction we have $\tilde{f} \in (G)^{(l)}$ for some $l < k$, so we can write
\begin{equation}\label{eq:lowerorder}
\tilde{f} = \sum_{j = 1}^{K_1} \alpha_{1,j}(y)\partial^{\sigma_{1,j}}\left(\tilde{f}_1\right) + \ldots + \sum_{j = 1}^{K_r}\alpha_{r,j}(y)\partial^{\sigma_{r,j}}\left(\tilde{f}_r\right)
\end{equation}
where the $\alpha_{i,j} \in K\{y\}$ and  $\ord \partial^{\sigma_{i,j}}\Le l < k$.

To complete the proof, we need the following fact about systems of homogeneous degree 1 polynomials.  Suppose $p,p_1,\ldots,p_s \in K[X_1,\ldots,X_n]$ are homogeneous degree 1 polynomials.  If there exist $q_1,\ldots,q_s \in K[X_1,\ldots,X_n]$ such that $p = q_1p_1 + \ldots + q_sp_s$, then we can in fact assume that all of the $q_i$ are constant.  Indeed, write $p = a_1X_1 + \ldots + a_nX_n$.  Assume without loss of generality that $a_n \neq 0$.  Since $p = q_1p_1 + \ldots + q_sp_s$, then $$X_n = q_0 + \frac{q_1}{a_n}p_1 + \ldots + \frac{q_s}{a_n}p_s,\quad q_0 := -\frac{a_1}{a_n}X_1 - \ldots - \frac{a_{n-1}}{a_n}X_{n-1}.$$  Thus, it suffices to prove the result when $p = X_n$.

For this, we order the variables so that $X_1 > \ldots > X_n$.  
Applying the Gauss-Jordan elimination to the system $\{p_i=0\}$, we obtain a new system $\{p_i'=0\}$
that is in reduced row echelon form.  
Moreover, every $p_i'$ is a linear combination of $p_1,\ldots,p_r$ (with coefficients in $K$) and vice versa.  There are two cases to consider.  If $X_n$ is a leading variable in $\{p_i'=0\}$, then because of the ordering on the $X_i$, we must have in fact that $X_n$ is one of the $p_i'$, and so the proof is complete.  

Therefore, suppose $X_n$ is not a leading variable of $\{p_i'=0\}$.  By assumption, $X_n \in (p_1,\ldots,p_s) = (p_1',\ldots,p_s')$.  This implies that for every solution $(\alpha_1,\ldots,\alpha_n)$ of the system $\{p_i = 0\}$ (or equivalently $\{p_i' = 0\}$),  $\alpha_n = 0$.  Thus, $X_n$ cannot be a free variable of $\{p_i'=0\}$, since there is a solution of the system $\{p_i' = 0\}$ for every possible value of any free variable (provided that a solution exists, which in this case is true, given by $(0,\ldots,0)$). 

Now, since the $\partial^{\gamma_j}y$ and $\partial^{\sigma_{i,j}}(\tilde{f}_i)$ in \eqref{eq:lowerorder} are all homogeneous of degree 1, by the above discussion, we can assume that the $\alpha_{i,j}$ are all constants $b_{i,j} \in K$, so we obtain
\begin{equation}\label{eq:contradictioneq}
\tilde{f} = \sum_{j = 1}^{K_1} b_{1,j}\partial^{\sigma_{1,j}}\left(\tilde{f}_1\right) + \ldots + \sum_{j = 1}^{K_r}b_{r,j}\partial^{\sigma_{r,j}}\left(\tilde{f}_r\right).
\end{equation}
Let
$$
h_i = \sum_{j = 1}^{K_i} b_{i,j}X^{\sigma_{i,j}}.
$$
Based on our construction of \eqref{eq:polytodiff} we can go backwards and deduce, using \eqref{eq:contradictioneq},
that $f = h_1f_1 + \ldots + h_rf_r$.  Since we know that $\ord \partial^{\sigma_{i,j}} \Le l$, this means that  
$\deg h_i\Le l$, $1\Le i \Le r$, contradicting the fact that the maximum degree must be at least $k > l$.
\end{proof}

\begin{rem}
If $f = 1$ in Theorem \ref{theo:polytodiff}, then $\tilde{f} = y$.  Thus, by considering the system $G_1 = \{G,1-ty\} \subset K\{t,y\}$, we have $1 \in (G_1)^{(k)}$ and $1 \notin (G_1)^{(k-1)}$.
\end{rem}

\begin{ex}\label{ex:lowerbound} 
For $m \Ge 2$, consider the following system of polynomial equations in $K[X_1,\ldots,X_m]$; cf.~\cite[page~578]{Brownawell}:
$$
f_1 = X_1^h, f_2 = X_1 - X_2^h, \ldots, f_{m-1} = X_{m-2} - X_{m-1}^h, f_m = 1-X_{m-1}X_m^{h-1}.
$$
It is shown that $1 \in (f_1,\ldots,f_m)$ and if $1 = g_1f_1 + \ldots + g_mf_m$, then $$\deg(g_1) \Ge h^m - h^{m-1} = h^{m-1}(h-1).$$  Thus, if $k$ is the maximum degree of the $g_i$ (that is smallest possible over the collection of all $g_i$ so that $1 = \sum g_if_i$), we must have that $k \Ge h^{m-1}(h-1)$.

Let us use this polynomial system to create a system of differential polynomial in $K\{y\}$ with derivations $\Delta = \{\partial_1,\ldots,\partial_m\}$.  Let $G$ be the system in $K\{y\}$ given by
\begin{equation}\label{eq:F1}
\tilde{f}_1 = \partial_1^hy, \tilde{f}_2 = \partial_1y-\partial_2^hy, \ldots, \tilde{f}_{m-1} = \partial_{m-2}y-\partial_{m-1}^hy, \tilde{f}_m = y-\partial_{m-1}\partial_m^hy.
\end{equation}
By the above discussion, we have $y \in (G)^{(k)}$ where $k \Ge h^{m-1}(h-1)$ and $y \notin (G)^{\left(h^{m-1}(h-1)-1\right)}$.  We have thus constructed a linear system $G$ in which the number of derivations of the elements of $G$ needed is exponential in the number of derivatives and polynomial in the order of the system.

We can construct an explicit linear combination of the $\tilde{f}_i$s and their derivatives equaling $y$ that requires exactly $h^{m-1}(h-1)$ derivations of $\tilde{f}_1$.  Explicit $g_i$s are constructed in~\cite{Brownawell} such that $1 = g_1f_1 + \ldots + g_mf_m$ and $\deg(g_1) = h^{m-1}(h-1)$ by observing that, setting $D = h^{m-1}(h-1)$, 
\begin{equation}\label{eq:polysys}
X_m^D\left(X_1^h\right) - \sum_{i = 2}^{m-1} X_m^D\left(X_{i-1}^{h^{i-1}} - \left(X_i^h\right)^{h^{i-1}}\right) + \left(1-\left(X_{m-1}X_m^{h-1}\right)^{h^{m-1}}\right) = 1.
\end{equation}
Thus, if we set
\begin{align*}
g_1 &= X_m^D \\
g_i &= X_m^D\left(\sum_{j = 0}^{h^{i-1}-1} X_{i-1}^{h^{i-1}-1-j}\left(X_i^h\right)^j\right) \hskip.5cm 2 \Le i \Le m-1 \\
g_n &= \sum_{j = 0}^{h^{m-1}-1} \left(X_{m-1}X_m^{h-1}\right)^j,
\end{align*}
then using \eqref{eq:polysys}, we have $1 = g_1f_1 + \ldots + g_mf_m$.

We can use these $g_i$s to find the desired linear combination of the $\tilde{f}_i$s and their derivatives.  Using the corresponding identities in $K[X_1,\ldots,X_m]$, we obtain that
\begin{align*}
\partial_{i-1}^{h^{i-1}}y - \partial_i^{h^i}y &= \sum_{j = 0}^{h^{i-1}-1}\partial_{i-1}^{h^{i-1}-j-1}\partial_i^{hj}\left(f_i\right) \hskip.5cm 2 \Le i \Le m-1 \\
y - \partial_{m-1}^{h^{m-1}}\partial_m^{h^{m-1}(h-1)}y &= \sum_{j = 0}^{h^{m-1}-1} \partial_{m-1}^j\partial_m^{j(h-1)}\left(f_m\right).
\end{align*}
Thus, setting $D = h^{m-1}(h-1)$, we can directly adapt \eqref{eq:polysys} to see that
\begin{equation}\label{eq:diffpolysys}
\partial_m^D\left(\partial_1^hy\right) - \sum_{i = 2}^{m-1} \partial_m^D\left(\partial_{i-1}^{h^{i-1}}y - \partial_i^{h^i}y\right) + \left(y - \partial_{m-1}^{h^{m-1}}\partial_m^{h^{m-1}(h-1)}y\right) = y.
\end{equation}
This gives us a linear combination of the $\tilde{f}_i$s and their derivatives that requires exactly $h^{m-1}(h-1)$ derivations of $\tilde{f}_1$, which we know is minimal by the polynomial case.
\end{ex}

\begin{ex}\label{ex:46}
We can use \eqref{eq:diffpolysys} to generalize this result to the case of multiple variables.  We will define a system in $K\{y_1,\ldots,y_n\}$ with derivatives $\Delta = \{\partial_1,\ldots,\partial_m\}$.  Let $m \Ge 2$. For $n=1$, we have~\eqref{eq:F1}. For $n \Ge 2$,  consider the collection of differential polynomials:
\begin{align*}
G_1 &= \left\{\partial_1^h y_1, \partial_1y_1 - \partial_2^hy_1,\partial_2y_1 - \partial_3^hy_1,\ldots,\partial_{m-2}y_1 - \partial_{m-1}^hy_1\right\} \\
G_i &= \left\{\partial_{m-1}y_{i-1} - \partial_1^hy_i,\partial_1y_i - \partial_2^hy_i,\ldots,\partial_{m-2}y_i - \partial_{m-1}^hy_i\right\} \hskip.5cm 2 \Le i \Le n-1 \\
G_n &= \left\{\partial_{m-1}y_{n-1}-\partial_1^hy_n,\partial_1y_n - \partial_2^hy_n,\ldots,\partial_{m-2}y_n - \partial_{m-1}^hy_n, y_n - \partial_{m-1}\partial_m^hy_n\right\}.
\end{align*}
Then let $G = \bigcup_{i = 1}^n G_i$.  We claim that $y_n \in (G)^{(k)}$ where $k \Ge h^{n(m - 1)}(h-1)$ and 
\begin{equation}\label{eq:yn}
y_n \notin (G)^{\left(h^{n(m-1)}(h-1) - 1\right)}.
\end{equation}  We can write a system similar to the one in \eqref{eq:diffpolysys} to produce $y_n$ in terms of the elements of $G$ and their derivatives needing exactly $h^{n(m-1)}(h-1)$ derivations of $\partial_1^hy_1$.  Let $E = h^{n(m-1)}(h-1)$.  Then we have
\begin{multline*}
\partial_m^E\left(\partial_1^hy_1\right) - \sum_{j = 1}^n \sum_{i = 2}^{m-1} \partial_m^E\left(\partial_{i-1}^{h^{(j-1)(m-1) + i - 1}}y_j - \partial_i^{h^{(j-1)(m-1)+i}}y_j\right) \\
- \sum_{j = 1}^{n-1}\partial_m^E\left(\partial_{m-1}^{h^{j(m-1)}}y_j - \partial_1^{h^{j(m-1) + 1}}y_{j+1}\right) + \left(y_n - \partial_{m-1}^{h^{n(m-1)}}\partial_m^{h^{n(m-1)}(h-1)}y_n\right) = y_n.
\end{multline*}

By the same argument that shows the minimality of $h^{m-1}(h-1)$ in Example \ref{ex:lowerbound}, we know that we must differentiate $\partial_1^hy_1$ at least $E$ times.  This shows~\eqref{eq:yn} and there is a $k \Ge E$ with $y_n \in (G)^{(k)}$.
\end{ex}

\setlength{\bibsep}{1.3pt}
\bibliographystyle{abbrvnat}
\small
\bibliography{bibdata}

\end{document}